\documentclass{amsart}
\usepackage{amssymb,amsmath,amsfonts}

\newtheorem{theorem}{Theorem}[section] 
\newtheorem{definition}[theorem]{Definition}
\newtheorem{proposition}[theorem]{Proposition}
\newtheorem{lemma}[theorem]{Lemma}

\theoremstyle{definition}

\newtheorem{example}[theorem]{Example}

\def\A{\mathbb{A}}
\def\cf{\mathcal{F}}
\def\cg{\mathcal{G}}
\def\ci{\mathcal{I}}
\def\cA{\mathcal{A}}
\def\cL{\mathcal{L}}
\def\k{\ensuremath{\mathbf{k}}}

\def\M{\mathbb{M}}
\def\P{\mathbb{P}}

\def\R{\mathbb{R}}
\def\U{\mathbb{U}}
\def\Z{\mathbb{Z}}
\def\O{\mathcal{O}}

\def\<{\ensuremath{\langle}}
\def\>{\ensuremath{\rangle}}

\def\excise#1{}

\DeclareMathOperator{\Trop}{Trop} 
\DeclareMathOperator{\Crem}{Crem}

\DeclareMathOperator{\Trunc}{Trunc}
\DeclareMathOperator{\codim}{codim}
\DeclareMathOperator{\Star}{Star}

\begin{document}

\title[Log-Concavity and the Bergman fan]{Log-concavity of characteristic polynomials and the Bergman fan of matroids}

\author[Huh]{June Huh}
\author[Katz]{Eric Katz}
\address{Department of Mathematics \\ University of Michigan \\
Ann Arbor \\ MI 48109}
\email{junehuh@umich.edu}
\address{Eric Katz: Department of Combinatorics \& Optimization \\ University of Waterloo \\ Waterloo \\ ON \\ Canada N2L 3G1}
\email{eekatz@math.uwaterloo.ca}

\begin{abstract}
In a recent paper, the first author proved the log-concavity of the coefficients of the characteristic polynomial of a matroid realizable over a field of characteristic $0$, answering a long-standing conjecture of Read in graph theory.  We extend the proof to all realizable matroids, making progress towards a more general conjecture of Rota-Heron-Welsh.  Our proof follows from an identification of the coefficients of the reduced characteristic polynomial as answers to particular intersection problems on a toric variety.  The log-concavity then follows from an inequality of Hodge type.
\end{abstract}

\maketitle


\section{Introduction}

In a recent paper \cite{Huh}, the first author proved that if $\M$ is a rank $r+1$ matroid realizable over a field of characteristic $0$ with characteristic polynomial,
\[
\chi_\M(q)=\mu_{0}q^{r+1}-\mu_{1}q^{r}+\cdots+(-1)^{r+1}\mu_{r+1}
\]
then the sequence $\mu_0,\dots,\mu_{r+1}$ is log-concave, that is, for $1 \le i \le r$,
\[
\mu_{i-1}\mu_{i+1}\leq \mu_i^2.
\]
Because graphic matroids are realizable over any field, this result proved a conjecture due to Read \cite{Read} that chromatic polynomials of graphs are unimodal.  There is a more general conjecture of Rota-Heron-Welsh \cite{Rota} that the coefficients of the characteristic polynomial of any finite matroid form a log-concave sequence.  The purpose of this paper is to extend the result of the first author to include all realizable matroids and to give some hints to an approach for proving the Rota-Heron-Welsh conjecture in general. A nice overview of the conjecture can be found in \cite{Aigner}.

Let us explain the first author's proof and our extension.  His proof uses a Morse-theoretic argument to relate $\mu_i$ to Milnor numbers of the singularity at the origin of a hyperplane arrangement with matroid $\M$.  These numbers are mixed multiplicites and are log-concave by the Khovanskii-Teissier inequality \cite[Example 1.6.4]{Lazarsfeld}.   Our method in this paper is to interpret the numbers $\mu_i$ as intersection numbers and apply the Khovanskii-Teissier inequality in a more classical framework. \excise{The fact that these mixed multiplicities are, in fact, intersection numbers is elaborated in a very nice foundational paper of Kleiman-Thorup \cite{KT}.} To identify the coefficients as intersection numbers, we use the combinatorial interpretation of the intersection theory on toric varieties developed by Fulton-Sturmfels \cite{FS} and studied in the context of tropical intersection theory by Mikhalkin \cite{Mi06}, Allermann-Rau \cite{AR}, and the second author \cite{KTT,KTIT}.  We use the fact that there is an explicit Poincar\'{e} dual to a compactification of the complement of a hyperplane arrangement in a particular toric variety. The Poincar\'{e} dual arises from the description of the Bergman fan studied by Ardila-Klivans \cite{AKBergman}. 

Let $\mathcal{A}$ be an arrangement of hyperplanes on an $r$-dimensional projective subspace $V \subset \mathbb{P}^n$ realizing $\M$. Let $\widetilde{V} \subset \mathbb{P}^n \times \mathbb{P}^n$ be the closure of the graph of the Cremona transformation
\[
\text{Crem} : \mathbb{P}^n \dashrightarrow \mathbb{P}^n, \qquad (z_0:\cdots:z_n) \mapsto  (z_0^{-1}:\cdots:z_n^{-1}) 
\]
restricted to  $V \setminus \mathcal{A}$. 
$\widetilde{V}$ is a compactification of $V \setminus \mathcal{A}$ whose boundary is a divisor with normal crossings.  By virtue of the description of the class of $\widetilde{V}$, we have the following result:

\begin{theorem}\label{graph}
Write 
\[
\overline{\chi}_\mathbb{M}(q) := \chi_\M(q)/(q-1) =\sum_{i=0}^r (-1)^i\mu^i q^{r-i}.
\]
Then
\[
[\widetilde{V}] = \sum_{i=0}^r \mu^i [\mathbb{P}^{r-i} \times \mathbb{P}^i] \in A_r(\mathbb{P}^n \times \mathbb{P}^n)
\]
in the Chow homology group of $\mathbb{P}^n \times \mathbb{P}^n$.
\end{theorem}

The log-concavity of $\mu^0,\ldots,\mu^r$, and hence that of $\mu_0,\ldots,\mu_{r+1}$, follows from applying the Khovanskii-Teissier inequality to the irreducibility variety $\widetilde{V}$. Our proof is largely combinatorial except for establishing the Khovanskii-Teissier inequality in Lemma \ref{l:logconcave} which requires the work of Fulton-Sturmfels and a classical proof of Khovanskii-Teissier.  For that reason, we do not know if our proof can be extended to general matroids. 

There is a related conjecture of Welsh and Mason that the number $f_i$ of independent subsets of cardinality $i$ form a log-concave sequence for any matroid \cite{ Mason, Welsh}. Theorem \ref{graph} implies that $f_i$ of a realizable matroid form a log-concave sequence because $f_i$ of a matroid is the coefficient $\mu^i$ of the reduced characteristic polynomial of its free coextension: See \cite[Section 7.4]{BrylawskiBook}, \cite[Theorem 4.2]{Brylawski}, and also \cite[Proposition 3.3]{Lenz}. We refer to \cite{Lenz} for further implications of Theorem \ref{graph}.

We would like to thank Sang-il Oum, Sam Payne, David Speyer, and David Wagner for helpful discussions.

\section{Matroids}

Let $\M$ be a rank $r+1$ matroid on the set $E=\{0,\ldots,n\}$ with rank function $r$. The characteristic polynomial of $\M$ is defined to be
\[\chi_\M(q)=\sum_{F\in\cL_\M} \mu(\hat{0},F)q^{r+1-r(F)}\]
where $\cL_\M$ is the lattice of flats, $\hat{0}$ is the minimum of $\cL_\M$, and $\mu$ is the M\"obius function of $\cL_\M$.  We write
\[\chi_\M(q)=\mu_{0}q^{r+1}-\mu_1q^r+\dots+(-1)^{r+1}\mu_{r+1}.\]
If $\M$ is realizable, then there is an $r$-dimensional projective subspace $V\subset\P^n$ which represents $\M$, that is, for $I\subset E$, 
\[r(I)=\codim\big((V\cap F_I)\subseteq V\big)\]
where $F_I$ is the coordinate flat given by $z_i=0$ for $i\in I$. The coordinate hyperplanes of $\P^n$ restrict to $V$ and define a projective arrangement $\cA$ on $V$.

In the sequel, we will restrict ourselves to simple matroids.  \excise{A point $x\in E$ is a {\em loop} if $r(\{x\})=0$.  A pair of points $x,y\in E$ is {\em parallel} if neither $x$ nor $y$ is a loop and $r(\{x,y\})=1$.} Recall that $\M$ is {\em simple} if it has no loops or pairs of parallel points.  \excise{In this case, $\hat{0}=\emptyset$.}  If $\M$ is not simple, then we can replace it by $\widehat{\M}$, the associated combinatorial geometry of $\M$, a matroid obtained by deleting loops and contracting parallel points \cite[Section 3.2]{StanleyNotes}.  $\widehat{\M}$ is a simple matroid whose lattice of flats is isomorphic to that of $\M$.  Therefore, $\widehat{\M}$ has the same characteristic polynomial as $\M$.  If $\M$ is realizable over a field $\k$, then so is $\widehat{\M}$.  Therefore, by proving the simple case, we establish the log-concavity for all realizable matroids.  

\begin{definition}  Let $c\cA$ be the {\em cone} of $\cA$, an essential central arrangement on $\A^{r+1}$ obtained by pulling back $\cA$ by $\A^{r+1}\setminus\{0\}\rightarrow V$.  Let $d\cA$ be the {\em decone} of $\cA$, an affine arrangement on $\A^r$ obtained from $\cA$ by declaring the hyperplane labelled by $0\in E$ to be the hyperplane at infinity.
\end{definition}

We have the posets of flats $\cL_\M,\cL_{c\cA},\cL_{d\cA}$.  The first is ordered by inclusion, and the others are ordered by reverse inclusion.  $\cL_{c\cA}$ is a geometric lattice isomorphic to $\cL_\M$ while $\cL_{d\cA}$ is a meet-semilattice but not a lattice in general.

Note that $\chi_\M(q)$ is divisible by $q-1$ in $\Z[q]$.  
\begin{definition}
The {\em reduced characteristic polynomial} $\overline{\chi}_\M(q)$ is 
\[\overline{\chi}_\M(q)=\chi_\M(q)/(q-1).\]
Define the numbers $\mu^0,\mu^1,\dots,\mu^r$ by
\[\overline{\chi}_\M(q)=\sum_{i=0}^r (-1)^i\mu^iq^{r-i}.\]
\end{definition}

$\overline{\chi}_\M(q)$ is the characteristic polynomial of $d\cA$.  Note that the log-concavity of $\overline{\chi}_\M$ implies that of $\chi_\M$.

\begin{definition}
Let $\emptyset\subsetneq F_1\subsetneq F_2\subsetneq\cdots\subsetneq F_k$ be a $k$-step flag of flats in $\M$.
\begin{enumerate}
\item The flag is said to be {\em initial} if $r(F_i)=i$.
\item The flag is said to be {\em descending} if $\min(F_1)>\min(F_2)>\dots>\min(F_k)>0$.
\end{enumerate}
Write $S_k$ for the set of initial, descending $k$-step flags of flats.
\end{definition}
The condition $0\not\in F_k$ implies that the flag is, in fact, a flag in $d\cA$.

\begin{proposition} \label{p:coeff} We have the following expression for $\mu^k$,
\[\mu^k=|S_k|.\]
\end{proposition}

\begin{proof}
We use the fact that $\mu^k$ is given by 
\[\mu^k=(-1)^k\sum_{I\in{(\cL_{d\cA}})_k} \mu(\hat{0},I)\]
where the sum is over  rank $k$ flats.
As a consequence of Weisner's theorem \cite[Section 3.9]{StanleyBook}, we have the following equality for any $a \in I$ \cite[Theorem 3.10]{StanleyNotes},
\[\mu(\hat{0},I)=-\sum_{a\notin F \lessdot I} \mu(\hat{0}, F)\]
where $A\lessdot B$ means that $A\subset B$ and $r(A)=r(B)-1$.
Therefore, if $I$ is a rank $k$-flat, we can iterate this formula to obtain
\begin{eqnarray*}
\mu(\hat{0},I)&=&-\sum_{0\notin F_{k-1}\lessdot I} \mu(\hat{0},F_{k-1})\\
&=&+\sum_{0\notin F_{k-1}\lessdot I}\left(
\sum_{\min(F_{k-1})\notin F_{k-2}\lessdot F_{k-1}} \mu(\hat{0},F_{k-2})\right)\\
&=&(-1)^{k-1}\sum_{0\notin F_{k-1}\lessdot I} \left( \sum_{\min(F_{k-1})\notin F_{k-2}\lessdot F_{k-1}} \left(\cdots
\left( \sum_{\min(F_2)\notin F_1\lessdot F_2} (-1)\right)\right) \right).
 \end{eqnarray*}
 Therefore, we are counting initial $k$-step flags of flats $\emptyset \subsetneq F_1\subsetneq F_2\subsetneq\dots\subsetneq I$ 
 satisfying $\min(F_1)>\min(F_2)> \cdots > \min(F_{k-1})>\min(I)>0$.
 By summing over rank $k$ flats $I$, we obtain the theorem.
 \end{proof}

\begin{definition} For a matroid $\M$ of rank $r+1$ on $E$ and $k\leq r$, let the \emph{truncation} $\Trunc_k(\M)$ be the matroid on $E$ with rank function $r_k$ given by
\[r_k(I)=\min\big\{r(I),k+1\big\}.\]
\end{definition}

$\Trunc_k(\M)$ is a rank $k+1$-matroid.  If $\M$ corresponds to an $r$-dimensional projective subspace $V\subset\P^n$, $\Trunc_k(\M)$ corresponds to $V\cap W$ where $W$ is a sufficiently general $k$-dimensional subspace of $\P^n$.

\section{Intersection theory on toric varieties}

We review some notions from the theory of toric varieties.  A toric variety $X=X(\Delta)$ is defined by a rational fan $\Delta$ in $N_\R=N \otimes_\Z \R$ for a lattice $N \simeq \mathbb{Z}^n$. \excise{
$X(\Delta)$ is complete if and only if $|\Delta|=N_\R$ while
$X(\Delta)$ is smooth if and only if $\Delta$ is unimodular, that is, if every cone in $\Delta$ is generated by a subset of an integral basis for $N$.}
The $k$-dimensional torus-invariant closed subvarieties of $X$ are of the form $V(\sigma)$, as $\sigma$ varies over the codimension $k$ cones in $\Delta$.
We write $N_\sigma$ for the sublattice of $N$ generated by $\sigma \cap N$.

When $X$ is complete, the operational Chow cohomology $A^*(X)$ has a combinatorial description given by Fulton and Sturmfels \cite{FS}.   Let $\Delta^{(k)}$ denote the set of all cones in $\Delta$ of codimension $k$.  If $\tau\in\Delta^{(k+1)}$ is contained in a cone $\sigma\in\Delta^{(k)}$, let $v_{\sigma/\tau}\in N/N_\tau$ be the primitive generator of the ray $(\sigma+N_\tau)/N_\tau$\excise{in $\Star_\tau(\Delta)/N_\tau$}.

\begin{definition} A function $c:\Delta^{(k)}\rightarrow\Z$ is said to be a {\em Minkowski weight} of codimension $k$ if it satisfies the {\em balancing condition}, that is, for every $\tau\in\Delta^{(k+1)}$,
\[\sum_{\sigma\supset\tau} c(\sigma)v_{\sigma/\tau}=0\]
in $N/N_\tau$.
\end{definition}

The main result of \cite{FS} is that $A^k(X)$ is canonically isomorphic to the group of codimension $k$ Minkowski weights.   The correspondence between Chow cohomology classes and Minkowski weights is as follows: given $d\in A^k(X)$, define $c(\sigma)=\deg\big(d\cap [V(\sigma)]\big)$.  The content of the Fulton-Stumfels result is that Chow cohomology classes are determined by their values on orbit closures.  The balancing condition is a combinatorial translation of the fact that cohomology classes are constant on linear equivalence classes.

Taking the cup product and taking the degree of a zero-dimensional class can be described combinatorially. 
The cup product is given by the \emph{fan displacement rule}.  Let $c_1$, $c_2$ be Minkowski weights of codimension $k_1$, $k_2$ respectively, and $v$ be a generic (as described in \cite{FS}) vector in $N_\R$. Given $c_1,c_2$, we can take $v$ outside a finite union of proper subspaces of $N_\R$. Then
\[(c_1\cup c_2)(\gamma)=
\sum_{(\sigma_1,\sigma_2)\in\Delta^{(k_1)}\times\Delta^{(k_2)}}m_{\sigma_1,\sigma_2}^\gamma c_1(\sigma_1)c_2(\sigma_2)\]
where $m_{\sigma_1,\sigma_2}^\gamma$ are defined by
\[m_{\sigma_1,\sigma_2}^\gamma=
\begin{cases}
[N:N_{\sigma_1}+N_{\sigma_2}]&\text{if $\gamma\subset\sigma_1,\sigma_2$ and $\sigma_1\cap(\sigma_2+v)\neq \emptyset$},\\
\hspace{12mm} 0 &\text{if otherwise.}
\end{cases}
\]
The degree $\deg(c)$ of a zero-dimensional class $c\in A^n(X)$ is defined to be $c(0)$, the value of $c$ on the unique zero-dimensional cone $0$.  

There is a notion of Poincar\'{e} duality in the intersection theory of toric varieties.  Suppose $X$ is smooth and let $Y\subset X$ be a subvariety of dimension $r$. Define a function
\[
c : \Delta^{(n-r)} \to \mathbb{Z}, \qquad \sigma \mapsto \deg\big([Y] \cdot [V(\sigma)]\big).
\]
Then $c$ is a Minkowski weight, called the {\em associated cocycle} of $Y$. The class $c$ acts as a Poincar\'{e} dual to $Y$ in the following sense:

\begin{lemma}\label{assoccyc} \cite[Lemma 9.2]{KTT} \excise{Suppose $Y$ intersects the torus orbits of $X$ properly.} If $c$ is the associated cocycle of $Y$, then
\[c\cap [X]=[Y]\in A_{r}(X).\]
\end{lemma}

Let $T$ be the dense torus in $X$ and $Y^\circ=Y \cap T$. The associated cocycle of $Y$ is closely related to the tropicalization of $Y^\circ$. Recall that the \emph{tropicalization} $\Trop(Y^\circ)$ is the set of vectors $v \in N_\R$ such that the initial degeneration $\text{in}_v(Y^\circ)$ in $T$ is nonempty, which is the underlying set of a rational fan of pure dimension $r$, together with the tropical multiplicity function $m$ \cite{Spe05}. Proposition 2.2 of \cite{T} states that $Y$ intersects torus orbits of $X$ properly iff $\Trop(Y^\circ)$ is a union of cones of $\Delta$.  In this case, $\Trop(Y^\circ)$ is the closure union of cones in $\Delta$ on which $c$ is positive.  Each $r$-dimensional cone $\sigma$ of $\Delta$ has tropical multiplicity $c(\sigma)$.

The equivariant Chow cohomology ring with integer coefficients $A^*_T(X)$ is naturally isomorphic to the ring of integral piecewise polynomial functions on $\Delta$, and there is a canonical map to ordinary Chow cohomology with integer coefficients
\[
\iota^*:A_T^*(X)\rightarrow A^*(X)
\]
induced by inclusions of $X$ in the finite dimensional approximations of the Borel mixed space \cite{Edidin-Graham}. For $\alpha \in A^*_T(X)$ and $c \in A^*(X)$, we write $\alpha \cup c$ to mean $\iota^* \alpha \cup c$. 

A \emph{$T$-Cartier divisor} $\alpha$ is an integral piecewise linear function on $\Delta$ viewed as an element of $A_T^1(X)$.
If $c\in A^k(X)$ is a Minkowski weight, we may compute the cup product $\iota^*\alpha\cup c$ as an element of $A^{k+1}(X)$ by using a formula that first appeared in \cite{AR}:
for $\sigma\in\Delta^{(k)}, \tau\in\Delta^{(k+1)}$, let $u_{\sigma/\tau}$ be a vector in $N_\sigma$ descending to $v_{\sigma/\tau}$ in $N/N_\tau$; then the value of $\iota^*\alpha\cup c$ on a cone $\tau\in\Delta^{(k+1)}$ is
\[
(\iota^*\alpha\cup c)(\tau)=-\sum_{\sigma\in\Delta^{(k)}|\sigma\supset\tau} \alpha_\sigma(u_{\sigma/\tau})c(\sigma)+\alpha_\tau\left(\sum_{\sigma\in\Delta^{(k)}|\sigma\supset\tau}c(\sigma)u_{\sigma/\tau}\right)
\]
where $\alpha_\sigma$ (respectively $\alpha_\tau$) is the linear function on $N_\sigma$ (on $N_\tau$) which equals $\alpha$ on $\sigma$ (on $\tau$).
A $T$-Cartier divisor $\alpha$ is said to be {\em nef} if for every codimension $1$ cone $\tau\in\Delta^{(1)}$, we have $\iota^*\alpha(\tau)\geq 0$.  This says that the cohomology class $\iota^*\alpha$ is non-negative on any $1$-dimensional orbit closure.  This notation is appropriate because a $T$-Cartier divisor $\alpha$ induces a $T$-equivariant line bundle on $X$ whose first Chern class is nef if and only if $\alpha$ is nef.

We have the following version of the Khovanskii-Teissier inequality.

\begin{lemma} \label{l:logconcave} Let $X=X(\Delta)$ be a smooth complete toric variety over an algebraically closed field.  Let $c$ be Poincare-dual to an $r$-dimensional irreducible variety $Y\subset X$ and $\alpha_1,\alpha_2$ be nef $T$-Cartier divisors on $\Delta$.  Then the numbers 
\[a_i=(\iota^* \alpha_1^{r-i}\cup \iota^* \alpha_2^{i}\cup c)\cap [X]\]
form a log-concave sequence.
\end{lemma}

\begin{proof}
The piecewise linear functions $\alpha_1$, $\alpha_2$ induce $T$-equivariant line bundles on $L_1,L_2$ on $X$.  Because every curve in $X$ is algebraically equivalent to a union of $1$-strata, the non-negativity condition on $\alpha_j$ ensures that $c_1(L_j)$ is nef \cite[Theorem 6.3.12]{CLS}. Now
\[(\iota^*\alpha_1^{r-i}\cup \iota^*\alpha_2^{i}\cup c)\cap [X]=(\iota^*\alpha_1^{r-i}\cup \iota^*\alpha_2^{i})\cap [Y]\]
and the result follows from the classical Khovanskii-Teissier inequality \cite[Example 1.6.4]{Lazarsfeld}.
\end{proof}

We do not know a purely combinatorial condition on the Minkowski weight $c$ for the above lemma to hold.

\section{Bergman fans}

Let $V$ be an $r$-dimensional projective subspace of $\mathbb{P}^n$ over the field $\mathbb{C}$ of complex numbers. The \emph{amoeba} of $V$ is the set of all vectors of the form
\[
\big(\log|x_1|,\log|x_2|,\ldots,\log|x_n|\big) \in \mathbb{R}^n
\]
where $(x_1,\ldots,x_n)$ runs over all points of $V$ in the torus $(\mathbb{C}^*)^n$. The asymptotic behavior of the amoeba is given by an $r$-dimensional polyhedral fan in $\mathbb{R}^n$ called the \emph{Bergman fan} of $V$. The Bergman fan of a projective subspace $V$ depends only on the associated matroid. More generally, one can associate to an arbitrary matroid $\M$ its Bergman fan which reflects combinatorial properties of $\M$ \cite[Section 9.3]{StuSolving}.

We introduce the Bergman fans of matroids following the exposition of Katz-Payne \cite{KPReal}.
Let $\M$ be a rank $r+1$ matroid on the set $E=\{0,\ldots,n\}$.  Let $N$ be the lattice
\[
N = \Z^E / \< e_0 + \cdots + e_n \>.
\]
We pick coordinates on $N$ in such a way that $e_1,\dots,e_n$ are the standard unit basis vectors and $e_0=(-1,\dots,-1)$.
The \emph{matroid fan} $\Delta_\M$ is a simplicial fan in $N_\R$ that encodes the lattice of flats of $\M$. Ardila and Klivans introduced this fan in \cite{AKBergman} and called it the fine subdivision of the Bergman fan of the matroid; the fan is defined as follows. For a subset $I \subset E$, let $e_I$ be the vector 
\[
e_I = \sum_{e_i \in I} e_i
\]
in $N_\R$. The rays of the matroid fan $\Delta_\M$ are $\R_{\ge 0}e_F$ for proper flats $F$ of the matroid.  In general, the $k$-dimensional cones of the matroid fan, $\sigma_\cf$ are the non-negative spans of $\{ e_{F_1}, \ldots, e_{F_k} \}$ for $k$-step flags of proper flats $\cf=\{\emptyset \subsetneq F_1 \subsetneq \cdots \subsetneq F_k\}$.  Since each cone $\sigma_\cf$ in $\Delta_\M$ is spanned by a subset of a basis for the lattice $N$, the toric variety $X(\Delta_\M)$ is smooth.  Furthermore, since every flag of flats in a matroid can be extended to a maximal flag of proper flats of length $r$, the matroid fan $\Delta_\M$ is of pure dimension $r$.

\begin{example} \label{uniform matroid}
Let $\U_n$ be the uniform matroid on $\{0, \ldots, n\}$, the matroid in which every subset is a flat.  Then the matroid fan $\Delta_{\U_n}$ in $N=\R^{n+1}/(1, \ldots, 1)$ is the first barycentric subdivision of the fan corresponding to $\P^n$, and $X(\Delta_{\U_n})$ is the toric variety obtained from $\P^n$ by a sequence of blowups
\[
X(\Delta_{\U_n}) = X_{n-1} \rightarrow \cdots \rightarrow X_1 \rightarrow X_0 = \P^n,
\]
where $X_{i+1} \rightarrow X_i$ is the blowup along the strict transforms of the $i$-dimensional torus-invariant subvarieties of $\P^n$. The Cremona transformation
\[
\text{Crem} : \mathbb{P}^n \dashrightarrow \mathbb{P}^n, \qquad (z_0:\cdots:z_n) \mapsto  (z_0^{-1}:\cdots:z_n^{-1}) 
\]
induces multiplication by $-1$ on $N$. $\text{Crem}$ extends to an automorphism
\[
\text{Crem} : X(\Delta_{\U_n}) \rightarrow X(\Delta_{\U_n})
\]
of $X(\Delta_{\U_n})$ since $\Delta_{\U_n}$ is invariant under the multiplication by $-1$.
\end{example}

Note that the labeling $E = \{e_0, \ldots, e_n\}$ of the underlying set of the matroid $\M$ induces an inclusion of the matroid fan $\Delta_\M$ as a subfan of $\Delta_{\U_n}$.  Furthermore, the dense torus $T$ in $X(\Delta_\M)$ is naturally identified with the dense torus in $\P^n$.

Let $\mathcal{A}$ be an arrangement of hyperplanes on an $r$-dimensional projective subspace $V \subset \mathbb{P}^n$ realizing $\M$. We can identify $\Delta_\M$ as the tropicalization of the complement $V^\circ :=V \setminus \mathcal{A} = V \cap T$.

\begin{theorem} \cite[Section 9.3]{StuSolving} 
The tropicalization of $V^\circ$ is $\Delta_\M$.
\end{theorem}

Let $\widetilde{V}$ denote the closure of $V^\circ$ in $X=X(\Delta_{\U_n})$.
Since the underlying set of $\Delta_\M$ is a union of cones of $\Delta_{\U_n}$, $\widetilde{V}$ intersects torus orbits of $X$ properly.  Consequently, $\Delta_\M$ considered as a Minkowski weight of $X$ is the associated cocycle of $\widetilde{V}$.



\section{Intersection theory computations}

Let $\alpha=\min\{0,x_1,\dots,x_n\}$ be a piecewise linear function on $\R^n$.  Note that $\alpha$ is linear on each cone of $\Delta_{\U_n}$. It is nef because $\alpha$ takes the value $1$ or $0$ on each $1$-dimensional orbit closure of $X(\Delta_{\U_n})$.  
In fact, $\alpha$ corresponds to a line bundle $p^*\O(1)$ on $X(\Delta_{\U_n})$ where $p$ is the blowup $p : X(\Delta_{\U_n}) \to \mathbb{P}^n$.  We have the following lemma which is to be expected from our geometric description of truncation:

\begin{lemma} \label{l:ohone} Let $\M$ be a rank $r+1$ matroid on $E=\{0,\dots,n\}$.  Then,
\[\alpha\cup\Delta_\M=\Delta_{\Trunc_{r-1}(\M)}\]
in $A^*\big(X(\Delta_{\U_n})\big)$
\end{lemma}

\begin{proof}
The Minkowski weight $\alpha\cup\Delta_\M$ is supported on codimension $1$ cones in $\Delta_\M$.  They correspond to $(r-1)$-step flags of proper flats
\[
\cf=\{\emptyset \subsetneq F_1\subsetneq F_2\subsetneq \dots\subsetneq F_{r-1}\}.
\]
The cone $\sigma_\cf$ is contained in $\sigma_\cg$ iff the flag $\cg$ is obtained from $\cf$ by inserting a single flat.  Write this relation as $\cg\gtrdot\cf$.  This flat must be inserted between two flags $F_j\subset F_{j+1}$ where $r(F_{j+1})=r(F_j)+2$. Setting $F_r := E$, there is a unique choice of $j$ where this happens.  Suppose $\cg$ is obtained from inserting a flat $F$ between $F_j\subset F_{j+1}$.
 Let $u_{\cg/\cf}$ be an integer vector in $\sigma_\cg$ that generates the image of $\sigma_\cg$ in $N/N_{\sigma_\cf}$. We may choose $u_{\cg/\cf}$ to be $e_{F}$. The value of $\alpha \cup \Delta_\M$ on $\sigma_\cf$ is given by
\[
(\alpha \cup \Delta_\M)(\sigma_\cf)=-\sum_{\cg\gtrdot \cf} \alpha_{\mathcal{G}}(u_{\cg/\cf})+\alpha_{\mathcal{F}}\Big(\sum_{\cg\gtrdot \cf} u_{\cg/\cf}\Big)
\]
where $\alpha_{\mathcal{G}}$ (respectively $\alpha_\mathcal{F}$) is the linear function on $N_{\sigma_\mathcal{G}}$ (on $N_{\sigma_\mathcal{F}}$) which equals $\alpha$ on $\sigma_\mathcal{G}$ (on $\sigma_\mathcal{F}$).

We now compute the right-hand side.  In any case 
\[
\alpha_{\mathcal{G}}(u_{\cg/\cf})=\begin{cases} -1&\ \text{if}\ 0\in F\\ 0&\ \text{if otherwise}. \end{cases}
\]
Let $f$ be the number of flats that can be inserted between $F_j$ and $F_{j+1}$.  Because every element of $F_{j+1}\setminus F_j$ is contained in exactly one flat $F$, we have
\begin{eqnarray*}
\sum_{\cg\gtrdot\cf} \alpha_{\mathcal{G}}(u_{\cg/\cf})&=&\begin{cases}
-f&\ \text{if}\ 0\in F_j\\
-1&\ \text{if}\ 0\in F_{j+1}\setminus F_j\\
0&\ \text{if otherwise}.
\end{cases}
\end{eqnarray*}
and
\begin{eqnarray*}
\sum_{\cg \gtrdot \cf} u_{\cg/\cf}&=&e_{F_{j+1}}+(f-1)e_{F_j}.\\
\end{eqnarray*}
It follows from the latter equality that if $j<r-1$, then 
\begin{eqnarray*}
\alpha_{\mathcal{F}}\Big(\sum_{\cg \gtrdot \cf} u_{\cg/\cf}\Big)&=&
\begin{cases}-f&\ \text{if}\ 0\in F_j\\ 
-1&\ \text{if}\ 0\in F_{j+1}\setminus F_j\\
0&\ \text{if otherwise}.\end{cases}
\end{eqnarray*}
If $j=r-1$, then $e_{F_{j+1}}=e_E=0$, and we have 
\begin{eqnarray*}
\alpha_{\mathcal{F}}\Big(\sum_{\cg \gtrdot \cf} u_{\cg/\cf}\Big)&=&
\begin{cases}-f+1&\ \text{if}\ 0\in F_j\\ 
0&\ \text{if otherwise}.\end{cases}
\end{eqnarray*}
Putting everything together, we have
\[
(\alpha \cup \Delta_\M)(\sigma_\cf)=\begin{cases}1&\ \text{if}\ j=r-1\\
0&\ \text{if otherwise}.\end{cases}
\]
Therefore $\alpha \cup \Delta_\M$ is non-zero on exactly the top-dimensional cones in $\Delta_{\Trunc_{r-1}(\M)} $.
\end{proof}

The next proposition relates coefficients of the reduced characteristic polynomial to certain intersection products on $X(\Delta_{\U_n})$. If $\Delta$ is a weighted fan considered as a Minkowski weight, then $\Crem^*(\Delta)$ is the weighted fan whose cones are $-\sigma$ for each $\sigma\in\Delta$ where the weight of $-\sigma$ in $\Crem^*(\Delta)$ is equal to that of $\sigma$ in $\Delta$.

\begin{proposition}\label{p:computation}
The coefficients of $\overline{\chi}_\M(q)$ are given by 
\[
\mu^k=\deg\big(\Delta_{\Trunc_{n-k}(\U_n)} \cup \Crem^*(\Delta_{\Trunc_k(\M)})\big).
\]
\end{proposition}

\begin{proof}
We use the Fulton-Sturmfels fan displacement rule. The right-hand side is the sum of the structure constants $m^0_{\sigma,\tau}$ for top-dimensional cones $\sigma \in \Delta_{\Trunc_{n-k}(\U_n)}$ and  $\tau \in\Crem^*(\Delta_{\Trunc_k(\M)})$.  We will show that $m^0_{\sigma,\tau}$ is always equal to $0$ or $1$ and the set of pairs $(\sigma,\tau)$ with $m^0_{\sigma,\tau}=1$ can be put in bijective correspondence with $S_k$, the set of  initial, descending $k$-step flags of flats.

Top-dimensional cones in $\Delta_{\Trunc_{n-k}(\U_n)}$ are of the form $\sigma_\ci$ for a flag 
\[
\ci=\{\emptyset \subsetneq I_1 \subsetneq \cdots \subsetneq I_{n-k}\},\qquad |I_j|=j.
\]
Taking union of $\sigma_\ci$ over all $\ci$ as above, we see that the underlying set of $\Delta_{\Trunc_{n-k}(\U_n)}$ is exactly the set of points where the minimum of $\{0,x_1,\dots,x_n\}$ is achieved at least $k+1$ times. The acheived minimum on $\sigma_\ci$ is $0$ iff $I_{n-k}$ does not contain $0$. In this case, 
\[
N_{\sigma_\ci} = \text{Span}\{e_i \mid i \in I_{n-k}\}.
\]
Top-dimensional cones in $\Crem^*(\Delta_{\Trunc_k(\M)})$ are of the form $-\sigma_\cf$ for a flag 
\[
\cf=\{\emptyset \subsetneq F_1\subsetneq\dots\subsetneq F_k\}, \qquad r_k(F_i)=i.
\]
Fix a generic vector $v=(v_1,v_2,\dots,v_n)\in\R^n$ with $0<v_1<v_2<\dots<v_n$. The claimed equality follows from Proposition \ref{p:coeff} and the lemma below. 

\begin{lemma}
The following are equivalent.
\begin{enumerate}
\item $|\Delta_{\Trunc_{n-k}(\U_n)}| \cap (-\sigma_\cf+v)$ is a singleton set.
\item $|\Delta_{\Trunc_{n-k}(\U_n)}| \cap (-\sigma_\cf+v)$ is nonempty.
\item $\min(F_1)>\min(F_2)>\dots>\min(F_k)>0$.
\end{enumerate}
If one of the above holds, then
\[
N_{\sigma_\ci}+N_{-\sigma_\cf}=N
\]
for the unique top-dimensional cone $\sigma_{\ci}$ of $\Delta_{\Trunc_{n-k}(\U_n)}$ intersecting $-\sigma_\cf+v$.
\end{lemma}

\begin{proof}
Suppose $(-x+v)$ is an element of $|\Delta_{\Trunc_{n-k}(\U_n)}| \cap (-\sigma_\cf+v)$. We write 
\[
x=t_1e_{F_1}+\dots+t_ke_{F_k}, \quad t_i \ge 0.
\]
Let $s \in E$ be the element with $0\in F_s\setminus F_{s-1}$, where we set $F_0 =\emptyset$, $F_{k+1} = E$. Then for any $l \in F_j \setminus F_{j-1}$,
\[
x_l=\begin{cases}
t_j+\dots+t_{s-1} &\ \text{if}\ j<s\\
\hspace{10mm} 0 &\ \text{if}\ j=s\\
-t_s-\dots-t_{j-1}&\ \text{if}\ j>s.
\end{cases}
\]
Note that the minimum of $v_l-x_l$ as $l$ ranges among elements $F_j\setminus F_{j-1}$ is achieved uniquely by $l_j :=\min(F_j\setminus F_{j-1})$ by our choice of $v$. Therefore the minimum of $\{0,v_1-x_1,\dots,v_n-x_n\}$ can be achieved by at most one element from each set $F_j\setminus F_{j-1}$.  For the minimum to be achieved $k+1$ times, it must be achieved by $0$. It follows that $0 \notin F_k$ and $x_l=v_{l_j}$ for $l \in F_j \setminus F_{j-1}$. Since
\[
x_{l_1}=v_{l_1}>x_{l_2}=v_{l_2}>\cdots>x_{l_k}=v_{l_k},
\]
we must have $l_1>l_2>\cdots>l_k>0$. In other words,
\[
\min(F_1)>\min(F_2)>\dots>\min(F_k)>0.
\]

Conversely, suppose $l_1>l_2>\cdots>l_k>0$ so that $v_{l_j}$ is an increasing sequence. Let $x=(x_1,\ldots,x_n) \in \mathbb{R}^n$ be the point obtained by setting $x_l=v_{l_j}$ for $l \in F_j \setminus F_{j-1}$. Then $x$ is contained in $\sigma_{\cf}$ because
\[
x=(v_{l_1}-v_{l_2})e_{F_1}+\cdots+(v_{l_{k-1}}-v_{k})e_{F_{k-1}}+(v_{l_k})e_{F_k}.
\]
The above analysis shows that this $x$ is the unique element of $|\Delta_{\Trunc_{n-k}(\U_n)}| \cap (-\sigma_\cf+v)$. We have shown that the conditions (1), (2), and (3) are equivalent.

Let $\sigma_\ci$ be the cone of $\Delta_{\Trunc_{n-k}(\U_n)}$ corresponding to the flag
\[
\ci=\{\emptyset \subsetneq I_1 \subsetneq \cdots \subsetneq I_{n-k}\},\qquad |I_j|=j.
\]
If $\sigma_\ci$ intersects $-\sigma_\cf+v$, then the above argument shows that $0 \notin I_{n-k}$ and
\[
\{l_1,\ldots,l_k\} \cup I_{n-k} = \{1,\ldots, n\}.
\]
The span of $-e_{F_1},\ldots,-e_{F_k}$ in $N/N_{\sigma_\ci}$ is generated by
\[
-e_{l_1},-e_{l_1}-e_{l_2},\dots,-e_{l_1}-e_{l_2}-\dots-e_{l_k}.
\]
This gives all of $N/N_{\sigma_\ci}$, and hence $N_{\sigma_\ci}+N_{-\sigma_\cf}=N$.
\end{proof}
\end{proof}

\section{Log-Concavity}

In this section we establish the log concavity of the numbers $\mu^k$ by interpreting them as intersection numbers.

\begin{lemma}\label{l:formula} 
\[
\mu^{k}=\deg\big(\alpha^{r-k} \cup (\Crem^*\alpha)^{k} \cup \Delta_\M \big).
\]
\end{lemma}

\begin{proof}
We observe that $\Delta_{\U_n}=\Crem^*(\Delta_{\U_n})$ is the associated cocycle of  $X(\Delta_{\U_n})$. By Proposition \ref{p:computation} and Lemma \ref{l:ohone},
\begin{eqnarray*}
\mu^k&=&\deg\big(\Delta_{\Trunc_{n-k}(\U_n)} \cup \Crem^*(\Delta_{\Trunc_k(\M)})\big)\\
&=&\deg\big(\Crem^*(\Delta_{\Trunc_{n-k}(\U_n)}) \cup \Delta_{\Trunc_k(\M)}\big)\\
&=&\deg\big( \Crem^*(\alpha^k \cup \Delta_{\U_n}) \cup (\alpha^{r-k} \cup \Delta_{\M})\big)\\
&=& \deg\big(\alpha^{r-k} \cup (\Crem^*\alpha)^{k} \cup \Delta_\M \big).
\end{eqnarray*}
\end{proof}

\begin{theorem}\label{main} 
If $\M$ is realizable, then the numbers $\mu^k$ form a log-concave sequence.
\end{theorem}

\begin{proof}
Suppose $\M$ is realized by a projective subspace $V$ in $\P^n$ over a field $\k$.  Then $\M$ is realizable over the algebraic closure $\overline{\k}$. Note that $\Delta_{\M}$ is Poincar\'e-dual to the closure $\widetilde{V}$ of $V^\circ=V \cap T$ in $X(\Delta_{\U_n})$. Since $\Crem$ is an automorphism of $X(\Delta_{\U_n})$ and $\alpha$ is nef, $\Crem^*\alpha$ is nef. Now, Lemma \ref{l:logconcave} applies to the formula of Lemma \ref{l:formula}.
\end{proof}

\begin{proof}[Proof of Theorem \ref{graph}]
Let $\pi_1,\pi_2$ be the projection of $\mathbb{P}^n \times \mathbb{P}^n$ onto the first and the second factor respectively. Write $L_1,L_2$ for the pull-back of the line bundle $\mathcal{O}(1)$ on $\mathbb{P}^n$ by $\pi_1,\pi_2$ respectively. Note that $X=X(\Delta_{\U_n})$ is realized in $\mathbb{P}^n \times \mathbb{P}^n$ as the closure of the graph of $\Crem:\mathbb{P}^n \dashrightarrow \mathbb{P}^n$. With this identification, the closure $\widetilde{V}$ of $V^\circ$ in $X$ is the graph closure of $\Crem$ restricted to $V \setminus \cA$. Note that the pull-back of $L_1,L_2$ to $X$ is the line-bundle corresponding to $\alpha, \Crem^* \alpha$ respectively. Consequently,
\[
(L_1^{r-k} \cup L_2^k) \cap [\widetilde{V}] = (\alpha^{r-k} \cup (\Crem^*\alpha)^k \cup \Delta_\M) \cap [X] = \mu^k.
\]
Therefore
\[
[\widetilde{V}] = \sum_{i=0}^r \mu^i [\mathbb{P}^{r-i} \times \mathbb{P}^i] \in A_r(\mathbb{P}^n \times \mathbb{P}^n)
\]
in the Chow homology group of $\mathbb{P}^n \times \mathbb{P}^n$.
\end{proof}

\bibliography{math}
\bibliographystyle{amsalpha}

\end{document}